\documentclass[12pt]{amsart}    
\usepackage{graphicx}
\usepackage{amsthm}
\usepackage{amsmath,amsfonts,amssymb,mathrsfs}
\usepackage{color}
\usepackage{epstopdf}
\usepackage{subfigure}
\usepackage[utf8]{inputenc}
\usepackage{hyperref}
\usepackage{cleveref}
\usepackage{cite}
\usepackage{bookmark}
\newtheorem{theorem}{Theorem}[section]
\newtheorem{lem}[theorem]{Lemma}
\newtheorem{cor}[theorem]{Corollary}
\newtheorem{prop}[theorem]{Proposition}
\newtheorem{rem}[theorem]{Remark}
\newcommand{\M}{{\mathcal M}}

\newcommand{\HH}{{\mathcal H}}
\newcommand{\LL}{{\mathcal L}}
\newcommand{\RR}{{\mathcal R}}
\newcommand{\ZZ}{{\mathcal Z}}
\newcommand{ \CC}{{\mathcal C}}

\newcommand{\Ad}{{\rm Ad}}
\newcommand{\GL}{{\rm GL}}
\newcommand{\Aut}{{\rm Aut}}

\newcommand{\ol}{\overline}
\let\ap=\alpha
\newcommand{\C}{\mathbb C}

\newcommand{\Z}{\mathbb Z}
\newcommand{\N}{\mathbb N}

\newcommand{\G}{{\mathcal G}}
\newcommand{\Sc}{\mathcal S}

\begin{document}

\title{The structure of Cartan subgroups \hbox{in Lie groups}}
\author[A.\ Mandal]{Arunava Mandal}
\address{Theoretical Statistics and Mathematics Unit,
Indian Statistical Institute, Bangalore Centre, 
Bengaluru 560059, India.}  
\email{a.arunavamandal@gmail.com} 
\author[R.\ Shah]{Riddhi Shah}
\address{School of Physical Sciences, Jawaharlal Nehru University, New Delhi 110067
India}  
\email{riddhi.kausti@gmail.com, rshah@jnu.ac.in} 
	
	
	
\date{June 17, 2020}

\begin{abstract}
We study properties and the structure of Cartan subgroups in a connected Lie group. We obtain a characterisation of 
Cartan subgroups which generalises W\"ustner's structure theorem for the same. We show that Cartan subgroups are 
same as those of the centralizers  of maximal compact subgroups of the radical. Moreover, we describe a recipe for 
constructing Cartan subgroups containing certain nilpotent subgroups in a connected solvable Lie group. We characterise the 
Cartan subgroups in the quotient group modulo a closed normal subgroup as the images of the Cartan subgroups 
in the ambient group. We also study the density of the images of power maps on a connected Lie group and show that 
the image of any $k$-th power map has dense image if its restriction to a closed normal 
subgroup and the corresponding map on the quotient group have dense images.
\end{abstract}

\maketitle

\section{Introduction}
Let $G$ be a connected Lie group. A Cartan subgroup of $G$ is a maximal nilpotent subgroup $C$ of $G$   with the property 
that if $L$ is any closed normal subgroup of finite index in $C$, then $L$ has finite index in its own normalizer in $G$.
This definition was given by Chevalley in \cite{Che}. Cartan subgroups were studied  by Chevalley \cite[Ch.\ 6]{Che}, 
Borel \cite{Bo1} and many others in the context of connected algebraic groups, see Togo \cite{To1,To2} for connected 
linear groups  and also Goto \cite{Go} for connected subgroups of Cartan subgroups of Lie groups. Later, it was studied 
by many in a different context (see \cite{Ch, D-M, H-M, Ne, P-R, Wu2} and the references cited therein). Cartan subgroups play 
an important role in the study of the structure of Lie groups, the behaviour of power maps and also in determining whether the 
exponential map is surjective or has dense image. 

  We explore properties and the structure of Cartan subgroups. There are some known results about their structure; one of them 
  is the so called Levi decomposition for Cartan subgroups given by W\"ustner. Any connected Lie group has a Levi decomposition 
  $G=SR$, where $S$ is a Levi subgroup, a maximal connected semisimple subgroup of $G$, and $R$ is the radical, the largest 
  connected solvable normal subgroup of $G$. Given a Cartan subgroup $C$ of $G$, W\"ustner showed that there exists a 
Levi subgroup $S$ of $G$, such that $C=(C\cap S)(C\cap R)$, where $C\cap S$ is a Cartan subgroup of $S$.  
We also know that $C\cap R$ is connected   and centralizes $C\cap S$ (see Theorem \ref{Wustner}). Here we show that 
$Z_R(C\cap S)$, the centralizer of $C\cap S$ in $R$ is a closed connected subgroup (more generally, see Proposition \ref{cetralizer-cartan}), 
and that $C\cap R$ is in fact a Cartan subgroup of $Z_R(C\cap S)$. More generally, we get the following characterisation.	

\begin{theorem} \label{cartan-constr}
Let $G$ be a connected Lie group and let $G=SR$ be a Levi decomposition, where $S$ is a Levi subgroup and $R$ is the 
radical of $G$. Let $C_S$ be any Cartan subgroup of $S$. Then $Z_R(C_S)$, the centralizer of $C_S$ in $R$ is connected 
and for any Cartan subgroup $C_{Z_R(C_S)}$ of $Z_R(C_S)$, $C=C_SC_{Z_R(C_S)}$ is a Cartan subgroup of $G$. Moreover,  
every Cartan subgroup $C$ of $G$ arises in the above form for some Levi subgroup $S$ of $G$.
\end{theorem}

Theorem 3\,(ii) of Goto in \cite{Go} shows that any Cartan subgroup $C$ 
of a connected Lie group $G$ contains a maximal compact subgroup $K$ of the radical $R$. It follows from 
W\"ustner's structure theorem (Theorem \ref{Wustner}) that $K$ is central in
$C$. In the following theorem, we show that the centralizer of such $K$ in $G$ is connected and 
we also show with the help of Theorem \ref{cartan-constr} that the Cartan subgroups of 
$G$ are precisely those of the centralizers in $G$ of the maximal compact subgroups of $R$. 

\begin{theorem} \label{str-solv}
Let $G$ be a connected Lie group and let $R$ be the radical of $G$. For every maximal compact subgroup $T_R$ of $R$, 
 the centralizer $Z_G(T_R)$ of $T_R$ in $G$ is connected and every Cartan subgroup of $Z_G(T_R)$ is a Cartan subgroup 
 of $G$. Conversely,  every Cartan subgroup $C$ of $G$ is a Cartan subgroup of $Z_G(K)$, where $K$ is the maximal compact 
 subgroup of $R$ contained in $C$. 
 \end{theorem}

Using Theorems \ref{cartan-constr} and \ref{str-solv}, we get the following corollary which is a more refined version of
Theorem \ref{cartan-constr}. 

\begin{cor} \label{cartan-whole} Let $G$ be a connected Lie group and let $G=SR$ be a Levi decomposition of $G$, where 
$S$ is a Levi subgroup and $R$ is the radical of $G$. Let $C_S$ be a Cartan subgroup of $S$.  Then there exists a maximal 
compact subgroup $T_R$ of $R$ which centralizes $C_S$ and   for any such group $T_R$, $Z_R(C_ST_R)$ is connected and 
$C_SC_{Z_R(C_ST_R)}$ is a Cartan subgroup of $G$, where $C_{Z_R(C_ST_R)}$ is any Cartan subgroup of $Z_R(C_ST_R)$. 
Conversely, given a Cartan subgroup $C$ of $G$, there exist a Levi decomposition $G=S R$ and a maximal compact subgroup 
$T_R$ of $R$ contained in $C$   such that  $C=C_SC_{Z_R(C_ST_R)}$, where $C_S=C\cap S$   is a Cartan subgroup of $S$ 
and $C_{Z_R(C_ST_R)}=C\cap R$ is a Cartan subgroup of $Z_R(C_ST_R)$. 
\end{cor}

 We know from W\"ustner's structure theorem (see Theorem \ref{Wustner} below) 
that for a Cartan subgroup $C$ of $G$, there exists a Levi subgroup $S$ such that $C_S=C\cap S$ is a Cartan subgroup of $S$. 
The question arises whether we can relate $C\cap R$ to a Cartan subgroup of $R$. The following corollary shows that 
$C\cap R$ is contained in a Cartan subgroup $C_R$ of $R$ and that $C\cap R=C_R\cap Z_R(C_S)$, where $C_S$ is as above. 

\begin{cor} \label{solv-cartan} Let $G$ be a connected Lie group with the radical $R$ and let $G=S R$ be a Levi decomposition 
of $G$ for a Levi subgroup $S$. Then given a Cartan subgroup $C_S$ of $S$, there exists a Cartan subgroup $C_R$ 
of the radical $R$ such that 
$C=C_S(Z_R(C_S)\cap C_R)$ is a Cartan subgroup of $G$. Conversely, given a Cartan subgroup $C$ of $G$, there exist a 
Levi decomposition $G=S R$ and a Cartan subgroup $C_R$ of $R$ such that $C\cap R=C\cap C_R=Z_R(C_S)\cap C_R$ and 
$C=C_S(Z_R(C_S)\cap C_R)$, where $C_S=C\cap S$ is a Cartan subgroup of $S$.	
\end{cor}

Note that W\"ustner's structure theorem implies in particular that if $C$ is a Cartan subgroup of a connected Lie group $G$, then 
$CR/R$ is isomorphic to $C_SR/R$ and it is a Cartan subgroup of $G/R$. Conversely, Theorem \ref{cartan-constr} implies that 
any Cartan subgroup of $G/R$ is an image of a Cartan subgroup of $G$ as $G/R$ is isomorphic to $S/(S\cap R)$, where 
$S\cap R$ is discrete and it is central in $S$. It is therefore natural to ask the following question (Q1): Do Cartan subgroups 
carry over to Cartan subgroups  in the quotient groups modulo closed normal subgroups? More precisely, if $H$ is a closed 
normal subgroup of $G$ and $C$ is a Cartan subgroup of $G$, is $CH/H$ a Cartan subgroup of $G/H$? Conversely, one can 
ask the following question (Q2): Do all Cartan subgroups of $G/H$ arise as images of Cartan subgroups of $G$? The answer 
was completely or partially known for a few cases. Recall that the center of the group $G$ is contained in every Cartan subgroup 
$C$. Moreover, $C$ is a Cartan subgroup of $G$ if and only if $C/Z$ is a Cartan subgroup of $G/Z$ for any closed central 
subgroup $Z$ of $G$. The answer to (Q1) is known when $H$ is any of the subgroups $G_1=\ol{[G,G]}$ and 
$G_{n+1}=\ol{[G,G_n]}$, $n\in\N$, $n\geq 2$, i.e.\ if $C$ is a Cartan subgroup of $G$, then $CH=G$, and $CH/H=G/H$ is a 
Cartan subgroup of $G/H$ as above \cite[Lemma 9]{Wi}. 
	
The following theorem shows that the answer to both (Q1) and (Q2) is affirmative for all closed normal subgroups $H$ of $G$	

\begin{theorem}\label{quo-cartan}
Let $G$ be a connected Lie group and let $H$ be any closed normal subgroup of $G$. Then the following hold:
\begin{enumerate}
\item[{$(a)$}] If $C$ is a Cartan subgroup of $G$, then $CH/H$ is a Cartan subgroup of $G/H$.
			
\item[{$(b)$}] If $Q$ is a Cartan subgroup of $G/H$, then there exists a Cartan subgroup $C$ of $G$ such that $CH/H=Q$.
\end{enumerate}
\end{theorem}

If $G$ is connected Lie group which is solvable or compact, then all the Cartan subgroups of $G$ are connected and one can 
deduce Theorem \ref{quo-cartan} in these cases from analogous results  about Cartan subalgebras in a Lie algebra (see 
Bourbaki \cite[Ch.\ VII, \S\,2]{Bou}). However, the theorem for the general case is not known as Cartan subgroups of a connected 
Lie group need not be connected. 

To prove Theorems \ref{cartan-constr} and \ref{quo-cartan} and Corollary \ref{solv-cartan}, we use Proposition \ref{solv-c} in which we 
describe a recipe for constructing Cartan subgroups containing those nilpotent groups which complement the nilradical in a connected 
solvable Lie group. As a consequence of these results, we get Corollary \ref{cartan-subh}, which shows in particular that given a Cartan
subgroup $C$ of a connected Lie group $G$ and a closed connected normal subgroup $H$ of $G$, $C\cap H$ is contained in a 
Cartan subgroup of $H$.
	
We also obtain a result on the behaviour of power maps. For $k\in\N$, let $P_k:G\to G$ be defined 
as $P_k(x)=x^k$, $x\in G$. Let $\G$ denote the Lie algebra of a connected Lie group $G$ and $\exp: \G\to G$ be the exponential map. 
Recall that a connected Lie group $G$ is said to weakly exponential if $\exp(\G)$ is dense in $G$. From \cite{B-M}, it is known that 
$G$ is weakly exponential if and only if $P_k(G)$ is dense for all $k\in\mathbb N$. Here we explore the conditions under which 
$P_k(G)$ is dense in $G$; see \cite{B-M} and \cite{Ma1} for some results in this direction for connected Lie groups and \cite{Ma2} 
for disconnected real algebraic groups. It is known that $P_k(G)$ is dense in $G$ if and only if $P_k(C)=C$ for every Cartan subgroup 
$C$ of $G$ \cite[Theorem 1.1]{B-M}. Moreover, if $R$ is the radical of $G$, then $P_k(G)$ is dense in $G$   if and only if 
$P_k(G/R)$ is dense in $G/R$ \cite[Proposition 3.3]{B-M}. 
	
Hoffman and Mukherjea \cite{H-M} proved that for any closed normal subgroup $H$ of $G$, if both $H$ and $G/H$ are  weakly 
exponential, then $G$ is weakly exponential. One can ask an analogous question in the context of the power maps. Using some known 
results on Cartan subgroups, we get  a more general result in this context as follows. 

\begin{theorem}\label{power-map}
Let $G$ be a connected Lie group, and let $H$ be any closed normal subgroup of $G$. Let $k\in\N$ be fixed. If the images of both 
the maps $P_k:H\to H$ and $P_k:G/H\to G/H$ are dense in $H$ and $G/H$ respectively, then $P_k(G)$ is dense in $G$. 
\end{theorem}

  The above theorem extends Corollary 1.3 of \cite{Ma2} where it is proven for the case when $G$ and $H$ are Zariski connected 
  real algebraic groups. In particular, the theorem together with \cite[Corollary 1.3]{B-M} together imply the result on weak 
  exponentiality of \cite{H-M} mentioned above (see Corollary \ref{w.e}).
	
The paper is organised as follows. In \S\,2, we fix some notations and prove some necessary results about nilpotent groups. In \S\,3, 
we analyse the structure of Cartan subgroups. We describe a recipe for constructing Cartan subgroups containing certain nilpotent 
subgroups in a connected solvable Lie group (see Proposition \ref{solv-c}).   
We also state and prove some results about the centralizer of Cartan subgroups of a Levi subgroup in the radical and 
prove Theorems \ref{cartan-constr} and \ref{str-solv}   as well as Corollaries \ref{cartan-whole} and \ref{solv-cartan}.  
In \S\,4, we prove Theorem \ref{quo-cartan} about the Cartan subgroups in quotient groups. 
Theorem \ref{power-map} on dense images of power maps of a connected Lie group is proved in \S\,5. 
	
\section{Notations and preliminaries}
	
In this section, we fix some notations, which will be used throughout the paper. We also state and prove some elementary results 
about nilpotent groups.
	
  Let $G$ be a locally compact group with the identity $e$ and let $H$ be a subgroup of $G$. Let $H^0$ denote the connected component 
of the identity $e$ in $H$,  $[H,H]$ denote the commutator subgroup of $H$ and let $Z(H)$ denote the center of $H$. Let $N_G(H)$ 
(resp.\ $Z_G(H)$) denote the normalizer (resp.\ centralizer) of $H$ in $G$.   All these are characteristic subgroups of $G$ if $H$ is so. 
  For a subgroup $L$ of $G$, let $N_L(H)=N_G(H)\cap L$, the set of all elements in $L$ which normalize $H$.   Let $N_L^0(H)=(N_L(H))^0$.
Note that $N_L(H)$ and $N^0_L(H)$ are closed in $L$ if $H$ is closed in $G$. For a subset $B$ of $G$, 
let $\ol{B}$ denote the closure of $B$ in $G$. Note that if $B$ is a group, then $\ol{B}$ is also a group, and it is connected 
(resp.\ normal) if $B$ is connected (resp.\ normal). For a subgroup $H$ of $G$, let $Z_H(B)$ denote the centralizer of the set 
$B$ in $H$, which is the same as $Z_H(G_B)$, where $G_B$ is the group generated by $B$ in $G$. Note that
$Z_H(B)=Z_H(\ol{B})$, it is a subgroup of $H$ and, it is closed if $H$ is closed. Let $Z_H^0(B)=(Z_H(B))^0$. 
For any $x\in G$, let $Z_H(x)$ denote the subgroup $Z_H(\{x\})$. 

  For a Lie group $G$, let $\G$ denote the Lie algebra of $G$ and let $\exp:\G\to G$ denote the exponential map. 
For a closed subgroup $H$ of $G$, let $\HH$ denote the Lie subalgebra of $H$ in $\G$ which is the same as that of $H^0$ and, let 
$C_H$ denote a Cartan subgroup of $H$ when $H$ is connected. 

Throughout,   for a connected Lie group $G$, let $R$ denote the radical of $G$ and $N$ the nilradical of $G$. 
They are characteristic subgroups of $G$, $\ol{[R,R]}\subset N$ and $G/R$ is semisimple.
	
Now we recall the usual definition of Cartan subgroups and some well-known facts about their structure. 
	
Let $G$ be a connected Lie group with the Lie algebra $\G$. A Lie subalgebra $  \CC$ of $\G$ is said to be a Cartan subalgebra if 
it is nilpotent and the normalizer of $\CC$ in $\G$ is $\CC$ itself. Let $\CC$ be a Cartan subalgebra of $\G$,   $\CC_\C$ 
(resp.\ $\G_\C$) be the complexification of $\C$  (resp.\  $\G$) and let $\Delta$ be the set of roots of $\G_\C$ belonging to $\CC_\C$.   
Then $\G_\C= \CC_\C+\Sigma_{\alpha\in\Delta}\G_ \C^{\alpha}$. Recall that the normalizer $N_G(\CC)$ of $\CC$ is defined by
$N_G(\CC):=\{g\in G\mid{\rm Ad}(g)(\CC)=  \CC\}.$  Let $C(\CC):=
\{g\in N_G(\CC)\mid\alpha\circ{\rm Ad}(g)|_{\CC_\C}=\alpha \mbox{ for all }  \alpha\in\Delta\}.$ 
A closed subgroup $C$ of a Lie group $G$ is called a Cartan subgroup if its Lie algebra $ \CC$ is a Cartan subalgebra of $\G$ and 
$C(\CC)=C$. 

The definition of Cartan subgroups given by Chevalley in \cite[IV.5.1]{Che} is as follows:     
A subgroup $C$ of a group $G$ is said to be a {\it Cartan} subgroup  if the following hold:
\begin{enumerate}
\item[{${\rm (I)}$}] $C$ is a maximal nilpotent subgroup of $G$.
\item[{${\rm (II)}$}] If $L$ is a normal   subgroup of finite index in $C$, then $N_G(L)/L$ is finite.
\end{enumerate}

It has been shown by Neeb (see Appendix in \cite{Ne}) that the above definitions are equivalent for any connected Lie group. 	
We now recall the following theorem of W\"ustner about Levi decomposition of Cartan subgroups in a connected Lie group. 
	
\begin{theorem}\label{Wustner} {\rm \cite[Theorems 1.11 and 1.9]{Wu2}}
Let $G$ be a connected Lie group and let $R$ be the radical of $G$. If $C$ is a Cartan subgroup, then there is a Levi subgroup  $S$ in $G$  
such that $C=C_S(C\cap R)$ where $C_S :=C\cap S$ is a Cartan subgroup of $S$ and, $C_S$ and $(C\cap R)$ centralize each other.
 Moreover, $(C\cap R)$ is connected.
\end{theorem}

 A connected locally compact nilpotent group $M$ is Lie projective, and hence it admits a unique maximal compact subgroup, say $T_M$. 
Note that $T_M$ is connected and central in $M$ and, $M/T_M$ is a simply connected nilpotent Lie group. 
The following result is known, we include a proof for the sake of completeness. 
	
\begin{lem}\label{normalizer} The following statements hold:
\begin{enumerate}
\item[{$(i)$}] Let $M$ be an abstract nilpotent group. Suppose that $Q$ is a proper subgroup of $M$. Then $Q$ is a proper subgroup 
of its normalizer $N_M(Q)$. 
\item[{$(ii)$}] Let $M$ be a connected locally compact nilpotent group. Suppose $Q$ is a proper closed connected subgroup of $M$.  
Then $N_M(Q)/Q$ is not finite. Moreover, if the maximal compact $($central$)$ subgroup $T_M$ of $M$ is contained in $Q$, then 
$N_M(Q)$ is connected. 
\end{enumerate}
\end{lem}

\begin{proof}
$(i):$ Let $M$ be a nilpotent group and let $l(M)$ denote the length of the central series of $M$. Let $Q$ be any proper subgroup of $M$. 
We will prove the first statement by induction on $l(M)$. For $l(M)=1$, $M$ is abelian, and thus we have $M=N_M(Q)$. As $Q\ne M$, the 
statement holds for all $M$ with $l(M)=1$. For $n\in\N$, suppose the statement holds for all $M$ with $l(M)\leq n$. Let $M$ be such 
that $l(M)=n+1$. Note that $Q\subset QZ(M)\subset N_M(Q)$. So, if $Z(M)$ is not contained in $Q$, then the statement holds easily. 
Now suppose $Z(M)\subset Q$. Then $Q/Z(M)$ is a proper subgroup of the nilpotent group $M/Z(M)$. Since, $l(M/Z(M))=n$, 
by using the induction hypothesis, we conclude that $Q/Z(M)$ is a proper subgroup of   its normalizer $N_{M/Z(M)}(Q/Z(M))$. If 
$\pi:M\to M/Z(M)$ is the natural projection, then $\pi^{-1}(N_{M/Z(M)}(Q/Z(M)))=N_M(Q)$. This implies that $Q$ is a proper subgroup 
of $N_M(Q)$.

\medskip		
\noindent $(ii):$   Let $M$ and $Q$ be as in (ii).   We already know from (i) that 
$N_M(Q)/Q$ is nontrivial.  Let $T_M$ be as in the statement of (ii). If $T_M \not\subset Q$, 
then $T_MQ/Q$ is connected and infinite. Now suppose $T_M\subset Q$. It is enough to show that $N_M(Q)/Q$ is connected.
Since $M/T_M$ and $Q/T_M$ are simply connected nilpotent Lie groups and hence algebraic, the normalizer of $Q/T_M$ in $M/T_M$ 
is also an algebraic subgroup, and hence connected. As $T_M\subset Q$ is connected and $N_M(Q)/T_M$ is the normalizer of 
$Q/T_M$ in $M/T_M$, we get that $N_M(Q)$ is connected. \qed
\end{proof}

For a Lie group $G$ and   $\alpha\in \Aut(G)$,   the group of automorphisms of $G$, let $S_\ap$ denote the stabilizer of 
$\ap$ in $G$, defined as $S_\ap=\{x\in G\mid \ap(x)=x\}$. 
It is a closed subgroup of $G$.  If $G$ is a finite dimensional real vector space, then $S_\ap$ is a subspace and it is closed.   
We now make a useful observation showing that the same holds for a larger class of groups.

\begin{lem} \label{stabilizer}
Let $G$ be a simply connected nilpotent Lie group and let $\ap\in\Aut(G)$. Then $S_\ap$, the stabilizer of $\ap$ is connected. 
Moreover, for any $A\subset\Aut(G)$, the intersection $\cap_{\ap\in A} S_\ap$ is connected. 
\end{lem}

This trivially follows from the fact that as $G$ is simply connected and nilpotent, the exponential map from the Lie algebra $\G$ to $G$ 
is a homeomorphism. The lemma can also be deduced easily from \cite[Proposition 2.5 and Theorem 2.11]{R}.
	
\section{The structure of Cartan subgroups}

In this section,   we first describe a recipe for constructing Cartan subgroups containing certain nilpotent subgroups in a 
connected solvable Lie group. Then we prove some useful results for a connected Lie group about the centralizer of a 
Cartan subgroup of a Levi subgroup in the radical and prove Theorems \ref{cartan-constr} and \ref{str-solv} 
and Corollaries \ref{cartan-whole} and \ref{solv-cartan}.  
	
Note that Cartan subgroups of a solvable Lie group are connected \cite[Theorem 1.9\,(iii)]{Wu1}. The following proposition 
will be useful in proving Theorems \ref{cartan-constr}, \ref{quo-cartan} and Corollary \ref{solv-cartan} and it shows that certain 
nilpotent subgroups are contained in Cartan subgroups of a connected solvable Lie group. In particular, it would imply that 
for any Cartan subgroup $C$ of a connected Lie group $G$, $C\cap R$ is contained in a Cartan subgroup of $R$, where $R$ is 
the radical of $G$. 

\begin{prop} \label{solv-c} Let $G$ be a connected solvable Lie group, $N$ be the nilradical of $G$ and let $T$ be the largest compact 
connected central subgroup of $G$. If $L$ is a closed nilpotent subgroup of $G$ such that $G=LN$, then the following hold:
\begin{enumerate}
\item[{$(1)$}]  $N_N(L^0T)$ is connected and $N_G(L^0T)$ is connected and nilpotent.
\item[{$(2)$}]  $L$ is contained in a Cartan subgroup $C$ of $G$ such that    $L\subset N_G(L^0)\subset N_G(L^0T)\subset C$. 
\end{enumerate}
 If $M$ is a closed connected subgroup of $G$ such that $G=MN$, then the following holds:
\begin{enumerate}
 \item[{$(3)$}]  Any Cartan subgroup $C_M$ of $M$ is contained in a Cartan subgroup $C$ of $G$ such that $C_M=C\cap M$. 
 \end{enumerate}
\end{prop}

\begin{proof}   Let $L$ and $T$ be as above.   Recall that $N_N(L^0T)=N_G(L^0T)\cap N$, the set of all elements in $N$ which normalize 
$L^0T$. Since $G=LN$ and $G/N$ is connected, we get that $G=L^0N$. Since $T\subset N$ and it is contained in every Cartan subgroup
of $G$, to prove $(1-2)$ we may replace $L$ by $LT$ and assume that $T\subset L^0$.   Let $\pi:G\to G/T$ be the natural projection. 
Then $\pi^{-1}(N_{G/T}(L^0/T))=N_G(L^0)$ and $\pi^{-1}(N_{N/T}(L^0/T))=N_N(L^0)$. Therefore, $N_G(L^0)$ (resp.\ $N_N(L^0)$) is 
connected if and only if $N_{G/T}(L^0/T)$ (resp.\ $N_{N/T}(L^0/T)$) is connected. Also, $C$ is a Cartan subgroup of $G$ 
if and only if $C/T$ is a Cartan subgroup of $G/T$. Therefore, to prove $(1-2)$ we may replace $L$ and $G$ by $L/T$ and $G/T$ 
respectively, and assume that $N$ is simply connected. 

\medskip
\noindent{\bf Step 1:} Note that $N_G(L^0)=L^0N_N(L^0)$ and that $L\subset N_G(L^0)$. As $L^0$ and $N_N(L^0)$ are nilpotent and 
normalize each other, we get that $N_G(L^0)$ is a closed nilpotent subgroup. Now we show that $N_G(L^0)$ is connected.
 It is enough to show that $N_N(L^0)$ is connected. Here, $L^0\cap N$ is normal in $N_N(L^0)$. Let $x\in N_N(L^0)$ and $g\in L^0$. 
 Then $x^{-1}gxg^{-1}\in L^0\cap N$. Since $L^0$ is connected, we get that $x^{-1}gxg^{-1}\in (L^0\cap N)^0$. Let 
 $L'=(L^0\cap N)^0=(L\cap N)^0$ and let $N'=N_N(L')$. Then $N_N(L^0)\subset N'$ and $L^0$ normalizes $L'$ as well as $N'$. 
 Moreover, $gxg^{-1}L'=xL'$ for all $x\in N_N(L^0)$ and $g\in L^0$, i.e.\ the conjugation action of $L^0$ on $N_N(L^0)/L'$ is trivial. 
 As $N$ is simply connected and nilpotent and $L'$ is connected, By Lemma \ref{normalizer}\,(ii), $N'$ is connected.
  Therefore, $N'/L'$ is simply connected and as noted above, the conjugation action of $L^0$ keeps $L'$ and $N'$ invariant. 
 For $g\in L^0$, let $\ap_g$ be the automorphism of $N'/L'$ defined as $\ap_g(xL')=gxg^{-1}L'$ for all $x\in N'$.  
 We get from Lemma \ref{stabilizer} 
that $\cap_{g\in L^0}S_{\ap_g}$ is connected. This, together with the fact that $L'$ is connected, imply that 
$B:=\{x\in N'\mid gxg^{-1}\in xL'\mbox{ for all } g\in L^0\}$ is connected. From the above discussion, we have that 
$N_N(L^0)\subset B$. As $L'\subset L^0$, it is easy to see that $B\subset N_N(L^0)$, and hence $N_N(L^0)=B$ is connected. 
Therefore, $N_G(L^0)=L^0N_N(L^0)$ is connected and (1) holds.  

\medskip
\noindent{\bf Step 2:} Let $L_1=N_G(L^0)$. Then $L\subset L_1$, $G=L_1N$ and $L_1$ is a closed connected nilpotent 
subgroup of $G$ as in (1). As assumed before Step 1, $T$ is trivial and $N$ is simply connected. 
In view of (1), we get a sequence of closed connected nilpotent subgroups $L_k$, such that 
$L\subset L_1=N_G(L^0)$, $L_{k+1}=N_G(L_k)\supset L_k$, $k\in\N$. Since each $L_k$ is a closed connected Lie subgroup, 
there exists $n\in N$ such that $L_n=L_{n+1}=N_G(L_n)$. We claim that $L_n$ is a Cartan subgroup of $G$. Since $L_n=N_G(L_n)$, 
it follows from Lemma \ref{normalizer}\,(i) that $L_n$ is a maximal nilpotent group.  Moreover, $N_G(L_n)/L_n$ is trivial, and since $L_n$ 
is connected, it follows that $L_n$ is a Cartan subgroup of $G$. As $L\subset L_1=N_G(L^0)\subset L_n$, (2) holds. 

\medskip
\noindent{\bf Step 3:} Now we prove (3). Let $M$ be as (3) such that $G=MN$ and let $C_M$ be a Cartan subgroup of $M$. By 
\cite[Lemma 9]{Wi}, $M=C_M\ol{[M,M]}$. As $[M,M]\subset N$, $G=MN=C_MN$. Note that $C_M$ is a closed connected nilpotent 
group and from (2), there exists a Cartan subgroup $C$ of $G$ such that $C_M\subset N_G(C_M)\subset C$. Note that $C\cap M$ is 
nilpotent and $C_M\subset C\cap M$. Since $C_M$ is a maximal nilpotent subgroup of $M$, we have that $C_M=C\cap M$. \qed
\end{proof}

  Now we deduce the following criterion for Cartan subgroups of a connected solvable Lie group. 

\begin{cor} \label{solv-norm} Let $G$ be a connected solvable Lie group. Then 
a subgroup $C$ is a Cartan subgroup of $G$ if and only if $C$ is connected and nilpotent and $N_G(C)=C$. 
\end{cor}

\begin{proof} If $C$ is a Cartan subgroup of $G$, then $C$ is closed, connected and nilpotent, it contains the center of $G$. Moreover by 
\cite[Lemma 9]{Wi}, $G=C\ol{[G,G]}$ and, as $\ol{[G,G]}\subset N$, the nilradical of $G$, we get that $G=CN$. As $C$ is connected 
and a maximal nilpotent group, it follows from Proposition \ref{solv-c}\,(1) that $N_G(C)=C$. 

Conversely, suppose $C$ is a connected nilpotent subgroup of $G$ such that $N_G(C)=C$. Suppose there is a nilpotent subgroup 
$C'$ of $G$ such that $C\subset C'$. Then by Lemma \ref{normalizer}\,(i), $N_{C'}(C)\ne C$. This contradicts the assumption above. 
In particular, $C$ is a maximal nilpotent group. Since $C$ is connected, it does not admit a proper (closed normal) subgroup of finite index. 
Now by Chevalley's criterion, $C$ is a Cartan subgroup of $G$. \qed
\end{proof}

\begin{rem} 
We can now assert that $C$ is a Cartan subgroup of a connected solvable Lie group $G$ if and only if $C$ is nilpotent, $G=CN$ and
$N_G(C)=C$. For if $C$ is a Cartan subgroup, then $C$ is nilpotent and $G=CN$ and by Corollary \ref{solv-norm}, $N_G(C)=C$. 
Conversely, if $C$ is nilpotent and $G=CN$. Then by Proposition \ref{solv-c}\,(2), $C\subset C'$ for some Cartan subgroup $C'$ of $G$. 
Since $C'$ is nilpotent and $N_G(C)=C$, by Lemma \ref{normalizer}\,(i), $C=C'$.
\end{rem}

For a Levi subgroup $S$ of $G$ and a Cartan subgroup $C_S$ of $S$, we now study the topological properties of centralizers of
$C_S$ in closed connected subgroups, in particular, in the radical and the nilradical of $G$. We also study some structural aspects 
of these groups. We first prove a proposition in this direction. For closed subgroups $H$ and $L$ of $G$ such that $H\subset L$ and 
$H$ is normal in $L$, and a set $A\subset G$  whose elements normalize both $H$ and $L$, let 
$Z_{L/H}(A):=\{xH\in L/H\mid axa^{-1}H=xH \mbox{ for all } a\in A\}$ and  $Z_{L/H}(a):=Z_{L/H}(\{a\})$, 
if $a\in G$ normalizes both $H$ and $L$. Here $Z_{L/H}(A)$ is the closed subgroup of $L/H$ consisting of elements fixed by 
the conjugation action of $A$ on $L/H$. Let $Z^0_{L/H}(A)$ (resp.\ $Z^0_{L/H}(a)$) be the connected component of the identity in 
$Z_{L/H}(A)$ (resp.\ $Z_{L/H}(a)$). We will use certain well-known properties of Cartan subgroups of a connected semisimple 
Lie group; we refer the reader to \cite[\S\,1.4]{Wa}. 

\begin{prop} \label{centralizer-conn}
Let $G$ be a connected Lie group. Let $S$ be a Levi subgroup   and let $C_S$ be a Cartan subgroup of $S$. Let $H$ and $L$ be 
closed connected subgroups of $G$  such that $H\subset L$, $H$ is normal in $L$, and $C_S$ normalizes both $H$ and $L$. 
Let $\pi:L\to L/H$ be the natural projection. Then $[\pi^{-1}(Z_{L/H}(C_S))]^0=\pi^{-1}(Z^0_{L/H}(C_S))=Z^0_L(C_S)H$. 
\end{prop}
	
\begin{proof}
The first two groups in the assertion are obviously equal since $H$ is connected. Let $c\in C_S$ and let 
$M_c=\pi^{-1}(Z^0_{L/H}(c))$. Then $M_c$ is a closed connected subgroup of $L$. Let $\HH$, $\LL$ and 
$\M_c$ be the Lie algebras of  $H$, $L$ and $M_c$ respectively. Then $\HH\subset\M_c\subset \LL$ and they 
are $\Ad(c)$-invariant. Since $\Ad(c)$ is semisimple, we get that $\M_c=W_c\oplus\HH$ for some $\Ad(c)$-invariant vector 
subspace $W_c\subset \M_c$.   As $\M_c/\HH$ is a Lie algebra of 
$M_c/H=Z^0_{L/H}(c)$, we get that  $\Ad(c)$ acts trivially on $\M_c/\HH$. Since $W_c$ is $\Ad(c)$-invariant and 
$\M_c=W_c\oplus\HH$,  we get that $\Ad(c)|_{W_c}={\rm Id}$, i.e.\ $\Ad(c)(w)=w$ for all $w\in W_c$. 
Therefore, $\M_c=Z_\LL(\Ad(c))\HH$. As $H$ is connected and normal, we get that $Z^0_L(c)H$ is an open 
subgroup of $M_c$ and as $M_c$ is connected, $Z^0_L(c)H=M_c$. For $c\in C_S$, let $\ZZ_c=Z_\LL(\Ad(c))$. 
As $\Ad(C_S)$ is abelian, we get that $\Ad(c')$ keeps $\ZZ_c$ invariant for all $c,c'\in C_S$. 
		
Let $M=\pi^{-1}(Z^0_{L/H}(C_S))$. Then $M$ is a closed connected subgroup, $H\subset M\subset L$ and $M$ is normalized by 
$C_S$. Let $\M$ be the Lie algebra of $M$. As $H\subset M\subset M_c$, we get that $M=Z^0_M(c)H$, for every $c\in C_S$. Also, 
since $Z^0_L(C_S)\subset M$, we have that $Z^0_L(C_S)=Z^0_M(C_S)$ and $Z_\LL(\Ad(C_S))=Z_\M(\Ad(C_S))$. Let 
$\ZZ'_c=Z_\M(\Ad(c))=\ZZ_c\cap\M$, $c\in C_S$. Then $\M=\ZZ'_c\HH=Z_\M(\Ad(c))\HH$ for all $c\in C_S$. 

We show by induction the following statement: For any $n\in\N$, given any $n$ elements $c_1,\ldots, c_n\in C_S$, 
$\M=(\cap_{i=1}^n\ZZ'_{c_i})\HH$. As shown above, it holds for $n=1$. Now suppose it holds for any subset of $C_S$ consisting of 
$n-1$ elements, where $n\geq 2$. Let $c_1,\ldots, c_n\in C_S$ and let $\ZZ'=\cap_{i=1}^{n-1}\ZZ'_{c_i}$.   By the induction 
hypothesis, $\M=\ZZ'\HH$. As $\Ad(c_n)$ commutes with $\Ad(c)$ for every $c\in C_S$, $\Ad(c_n)$ keeps $\ZZ'$ invariant. 
As $\Ad(c_n)$ is semisimple and keeps $\ZZ'$ and $\ZZ'\cap\HH$ invariant, we get a vector space $W\subset \ZZ'$ such that 
$\ZZ'=W\oplus(\ZZ'\cap \HH)$, where $W$ is $\Ad(c_n)$-invariant. Therefore, $\M=W\HH$, where $W\subset \ZZ'$.  
Since $\Ad(c_n)$ acts trivially on $\M/\HH$ and keeps $W$ invariant, we get that $\Ad(c_n)$ acts trivially on $W$. 
Therefore, $\M=(\ZZ'\cap\ZZ'_{c_n})\HH=(\cap_{i=1}^n\ZZ'_{c_i})\HH$.
		
Now considering the dimension of $Z_\M(\Ad(C_S))$, we get that there exist $c_1,\ldots, c_n$ such that
$Z_\LL(\Ad(C_S))=Z_\M(\Ad(C_S))=\cap_{i=1}^n Z_\M(\Ad(c_i))$. Therefore, $\M=Z_\LL(\Ad(C_S))\HH$. As $H$ is 
connected and normal in $L$ and as $M$ is connected, we get that $M=Z^0_L(C_S)H$.  \qed 
\end{proof}
	
\begin{rem} \label{rem1} Proposition \ref{centralizer-conn} only uses the fact that $\Ad(C_S)$ is abelian and consists of 
semisimple elements. Hence Proposition \ref{centralizer-conn} holds for any subset $B\subset G$ instead of $C_S$, 
such that all the elements of $B$ normalize both $L$ and $H$, and $\Ad(B)$ consists of semisimple elements which commute
with each other. In particular it holds for $C_S^0$ instead of $C_S$. 
\end{rem}

The following proposition will be useful in proving Theorem \ref{cartan-constr}.
	
\begin{prop}\label{cetralizer-cartan}
Let $G$ be a connected Lie group. Let $S$ be a Levi subgroup, $R$ be the radical of $G$ and let   
$N$ be the nilradical of $G$. Let $C_S$ be a Cartan subgroup of $S$. Then the following hold:
\begin{enumerate}
\item[{$(1)$}] $Z_N(C_S)$ and $Z_N(C_S^0)$ are connected.
			
\item[{$(2)$}] $Z_R(C_S)$ is connected and $R=Z_R(C_S)N=Z^0_R(C_S)N$.
			
\item[{$(3)$}] Let $C_{Z_R(C_S)}$ be a Cartan subgroup of $Z_R(C_S)$. Then $R=C_{Z_R(C_S)}N$, 
                      $Z_R(C_S)=C_{Z_R(C_S)} Z_N(C_S)$ and $C_{Z_R(C_S)}\cap N$ is connected. 
\end{enumerate}
\end{prop}
	
\begin{proof}         
$(1):$  Let $T$ be the (unique) maximal compact subgroup of the nilradical $N$. Then $T$ is connected and central in $N$, and it is   
also the largest compact connected central subgroup of $G$. Note that $N/T$ is a simply connected nilpotent group.  
Since $N/T$ is simply connected and nilpotent, it follows from Lemma \ref{stabilizer} that 
$Z_{N/T}(C_S)=\cap_{c\in C_S} Z_{N/T}(c)$ is connected.
Since $T$ is connected and central, by Proposition \ref{centralizer-conn} we get that $Z^0_N(C_S)/T=Z_{N/T}(C_S)$. 
Therefore, $Z_N(C_S)/T=Z_{N/T}(C_S)$ and $Z_N(C_S)$ is connected. Similarly, we can prove 
using Remark \ref{rem1} that $Z_N(C_S^0)$ is connected. 

\medskip
\noindent $(2):$  Note that $\ol{[S,R]}$, the closed subgroup generated by $\{srs^{-1}r^{-1}\mid s\in S, r\in R\}$ is contained in $N$.
In particular, $Z_{R/N}(C_S)=R/N=Z^0_{R/N}(C_S)$. By Proposition \ref{centralizer-conn}, $R=Z^0_R(C_S)N=Z_R(C_S)N$. Now 
$Z_R(C_S)=Z_R^0(C_S)(Z_R(C_S)\cap N)=Z_R^0(C_S)Z_N(C_S)$ and it follows from (1) that $Z_R(C_S)$ is connected. 
	
\medskip
\noindent $(3):$ Let $Z:=Z_R(C_S)$. Since any Cartan subgroup of a connected solvable Lie group is connected \cite{Wu1}, 
we have that $C_{Z_R(C_S)}=C_Z$ is a connected nilpotent subgroup of $Z$. Also, $\ol{[Z,Z]}\subset\ol{[R,R]}\subset N$. 
Thus we get that $Z=C_Z\ol{[Z,Z]}\subset C_ZN$ \cite[Lemma 9]{Wi}, and hence that  $R=ZN=C_ZN$. This implies that 
$Z_R(C_S)=C_Z Z_N(C_S)$. Note that by (1), $Z\cap N=Z_N(C_S)$   is connected. Moreover $Z\cap N$ is a closed connected 
nilpotent normal subgroup of $Z$. As $C_Z\cap N=C_Z\cap (Z\cap N)$, using \cite[Theorem 1.9\,(i)]{Wu2} for the group $Z$, 
we get that $C_Z\cap N$ is connected. \qed 
\end{proof}
 
We now prove some lemmas to use later. Recall that a Cartan subgroup of a connected semisimple Lie group may not be 
connected or abelian but its connected component of the identity is abelian. The next lemma should be known and it follows 
from \cite[Theorem 1.4\,(iii)]{Wu2} easily. It is stronger than \cite[Lemma 3.5]{B-M}. We give a short proof for the sake of 
completeness. 

\begin{lem} \label{ss-cartan} 
		Let $G$ be a connected semisimple Lie group and let $C$ be a Cartan subgroup of $G$. Then $Z_G(C^0)=C$.
\end{lem} 

\begin{proof} Let $\G$ be a Cartan subalgebra of $G$ and let $\CC$ be a Lie subalgebra of $\G$ which is the Lie algebra 
of $C^0$. Since $\G$ is semisimple, $\CC$ is abelian \cite[Proposition 6.47]{Kn}, and by \cite[Theorem 1.4\,(iii)]{Wu2} we get that 
$C=\{g\in G\mid \Ad(g)(X)=X \mbox{ for all } X\in\CC\}$. This implies that $C=Z_G(C^0)$. \qed 
\end{proof}

The following is known for connected semisimple Lie groups \cite[Theorem 1.4.1.5]{Wa}. We extend the result to all 
connected Lie groups. 

\begin{lem} \label{finite-index}
Let $G$ be a connected Lie group, and let $C$ be a Cartan subgroup of $G$. Then $C/C^0Z(G)$ is finite.
\end{lem}
	
\begin{proof}
Let $S$ be a Levi subgroup such that $C=(C\cap S)(C\cap R)$, where $C_S=C\cap S$ is a Cartan subgroup of $S$ and 
$R$ is the radical of $G$. Then it is well-known and easy to see that $\Ad(C_S)$ is a Cartan subgroup of $\Ad(S)$. Since $\Ad(S)$
is semisimple and linear and has finite center, by \cite[Theorem 1.4.1.5]{Wa}, $\Ad(C_S)$ is abelian and has finitely many connected 
components. Therefore, $C^0_SZ(G)$ has finite index in $C_SZ(G)$. Note that $Z(G)\subset C=C_S(C\cap R)$ and $C\cap R$ is 
connected, and hence $Z(G)C^0=Z(G)C^0_S(C\cap R)$ has finite index in $C$. \qed
\end{proof}

The following lemma will be useful in the proof of Theorem \ref{cartan-constr}. Note that once Theorem \ref{cartan-constr} is proved, 
Lemma \ref{cetralizer-cartan1} will be valid for all Cartan subgroups $C$ as it would have the decomposition as described below. 

\begin{lem}\label{cetralizer-cartan1}
Let $G$ be a connected Lie group.   Let $S$ be a Levi subgroup and $G=S R$ be a Levi decomposition, where $R$ is the 
radical of $G$. Let $C_S$ be a Cartan subgroup of $S$, let $C_{Z_R(C_S)}$ be a Cartan subgroup of $Z_R(C_S)$ and let 
$C:=C_SC_{Z_R(C_S)}$.   Let $N$ be the nilradical of $G$. 
Then the following   equalities hold:
\begin{enumerate}
\item[{$(1)$}] 
\begin{enumerate}
\item $Z_N(C_S^0)=Z_N(C_S)$ and $Z_R(C_S^0)=Z_R(C_S)$  
\item $Z_S(C^0)=C_S$.
\item $Z_G(C_S^0)=Z_S(C_S^0)Z_R(C_S^0)=C_SZ_R(C_S)=C_SC_{Z_R(C_S)}Z_N(C_S)$. 
\end{enumerate}
			
\item[{$(2)$}] 
$Z_G(C^0) = Z_S(C^0)Z_R(C^0) = C_SZ_R(C^0)\\
\hphantom{Z_G(C^0)= Z_S(C^0)Z_R(C^0)} = C_SZ_R(C) = C_SZ(C\cap R)=C_SZ(C)$.
  
\end{enumerate}
\end{lem}

\begin{proof}
\noindent $(1$a$):$ By Proposition \ref{cetralizer-cartan}\,(3), $R=C_{Z_R(C_S)}N$, and $Z_R(C_S)\subset Z_R(C_S^0)$, and 
we get that $Z_R(C_S^0)=C_{Z_R(C_S)}Z_N({C_S}^0)$. Therefore, it is enough to prove that $Z_N(C_S^0)=Z_N(C_S)$. Note that
$\Ad(S)$ is an almost algebraic subgroup of $\GL(\G)$, i.e.\ a subgroup of finite index in an algebraic subgroup of $\GL(\G)$. 
Since $\Ad(N)$ consists of unipotent elements, it is algebraic, and hence $\Ad(SN)=\Ad(S)\Ad(N)$ is an almost 
algebraic subgroup of $GL(\G)$ (see \cite{Da}).  By \cite[Theorem 1.4.1.5]{Wa}, 
$\Ad(C_S)$ is abelian and $\Ad(C_S^0)$ is a subgroup of finite index in $\Ad(C_S)$. Since $\Ad(S)$ is almost algebraic, it is 
well-known that $\Ad(C_S)=A\cap \Ad(S)$, where $A$ is a Zariski connected abelian algebraic group (see for instance \cite{Ch}). Therefore, $\Ad(C_S^0)$ and $\Ad(C_S)$ are both almost algebraic and Zariski dense in $A$. 

We first show that $Z_{\Ad(N)}(\Ad(C_S^0))=Z_{\Ad(N)}(\Ad(C_S))$. One way inclusion is obvious. Let $x\in Z_{\Ad(N)}(\Ad(C_S^0))$. 
Then $\Ad(C_S^0)\subset Z_{\Ad(S)}(x)$, the latter group is almost algebraic and it is an intersection of some algebraic 
group and $\Ad(S)$. Since $\Ad(C_S^0)$ and $\Ad(C_S)$ both have the same Zariski closure $A$, it follows that 
$\Ad(C_S)\subset Z_{\Ad(S)}(x)$. Since this holds for all such $x$ as above, we have that 
$Z_{\Ad(N)}(\Ad(C_S^0))=Z_{\Ad(N)}(\Ad(C_S))$. 

We know that $N\cap Z(G)=Z^0(G)$; (see \cite[Theorem 2.3]{Wu1}, where the group $N$ is supposed to be connected). Therefore, 
$\Ad(N)$ is isomorphic to $N/Z^0(G)$, 
and we get from the above discussion that $Z_{N/Z^0(G)}(C^0_S)=Z_{N/Z^0(G)}(C_S)$. 
As $Z^0(G)$ is connected and central,  it follows from Proposition \ref{centralizer-conn} and Remark \ref{rem1} that 
$Z_N(C^0_S)=Z_N(C_S)$.  			

\medskip 
\noindent $(1$b$):$ Let $Z:=Z_R(C_S)$. Then $C_S\subset Z_S(C_Z)$. Now we have that
$Z_S(C^0)=Z_S(C_S^0C_Z)=Z_S(C_S^0)\cap Z_S(C_Z)=C_S\cap Z_S(C_Z)=C_S$, by Lemma \ref{ss-cartan}.

\medskip		
\noindent $(1$c$):$ Let $y\in Z_G(C_S^0)$. As $G=SR$, $y=sr$ for some $s\in S$ and $r\in R$. Then for any $x\in C_S^0$,
$x^{-1}srx r^{-1}s^{-1}=(x^{-1}sxs^{-1})(sx^{-1}rxr^{-1}s^{-1})=e$. We get that $x^{-1}sxs^{-1}=(sx^{-1}rxr^{-1}s^{-1})^{-1}\in S\cap R$.
Note that $S\cap R$ is discrete and it is a central subgroup of $S$. 
As $C_S^0$ is connected, $A_x=\{x^{-1}sxs^{-1}\mid x\in C_S^0\}$ is a connected subset of $S\cap R$ and it contains the 
identity $e$. As $S\cap R$ is discrete, the preceding assertion implies that $A_x=\{e\}$. Therefore, $s\in Z_S(C_S^0)$, and hence 
$r\in Z_R(C_S^0)$. That is, $Z_G(C_S^0)=Z_S(C_S^0)Z_R(C_S^0)$. As $S$ is semsimple, by 
Lemma \ref{ss-cartan}, we get that $Z_S(C_S^0)=C_S$. Now the last two equalities in (1c) follows from (1a) and 
Proposition \ref{cetralizer-cartan}\,(3). 	

\medskip		
\noindent $(2):$ Let $Z=Z_R(C_S)$ as in (1b).  Then $C_Z$ is connected, and $C=C_SC_Z$. Using (1c), we get that  
$$Z_G(C^0)=Z_G(C_S^0)\cap Z_G(C_Z)= C_S(Z_R(C_S)\cap Z_G(C_Z))= C_SZ_Z(C_Z).$$ 
Also, $Z_R(C)=Z_R(C_S)\cap Z_R(C_Z)=Z\cap Z_R(C_Z)=Z_Z(C_Z)$. By Corollary \ref{solv-norm}, we have that 
$Z_Z(C_Z)\subset C_Z$, and hence we get that $Z_R(C)=Z(C_Z)=Z(C\cap R)\subset Z(C)$, where $Z(C\cap R)$ 
(resp.\ $Z(C)$) is the center of $C\cap R$ (resp.\ $C$). Now from above, we have that $Z_G(C^0)=C_SZ(C\cap R)$. Also, 
$C_SZ_R(C)\subset C_SZ(C)$ and conversely, $C_SZ(C)\subset C\subset C_SZ_R(C)$. 
Hence $C_SZ_R(C)=C_SZ(C)=Z_G(C^0)$.   

Now the other equalities in (2) follow since we get using (1a-1b) that  
$C_SZ_R(C)\subset C_SZ_R(C^0)=Z_S(C^0)Z_R(C^0)\subset Z_G(C^0)=C_SZ_R(C)$.   \qed
\end{proof}		

\medskip	
\noindent{\it Proof of Theorem~\ref{cartan-constr}}. Recall that $G=S R$ is a Levi decomposition with $S$ as a Levi subgroup 
and $C_S$ is a Cartan subgroup of $S$.   By Proposition \ref{cetralizer-cartan}\,(2), we have that $Z_R(C_S)$, the centralizer 
of $C_S$ in the radical $R$ is connected. Let $C_{Z_R(C_S)}$ be any Cartan subgroup of $Z_R(C_S)$.
Let $Z:=Z_R(C_S)$. We want to prove that $C=C_SC_Z$ is a Cartan subgroup of $G$. Note that $C=C_SC_Z$ is
 nilpotent. Note also that $S\cap R$ is discrete and it is central in $S$, and hence it is contained in $C_S$. 
 Also, $C_S\cap R=S\cap R\subset Z_R(C_S)\cap C_S$. Therefore, $S\cap R$ is central in $Z=Z_R(C_S)$, and hence it is 
 contained in $C_Z$. This implies that $C\cap R=(S\cap R)C_Z=C_Z$. 
 
 Now we show that $C$ is a maximal nilpotent subgroup of $G$. Suppose that $M$ is a nilpotent subgroup of $G$ containing $C$.   
 We want to show that $M=C$. As $G/R$ is isomorphic to $S/(S\cap R)$, and the latter has $C_S/(S\cap R)$ as a Cartan 
 subgroup, $C_SR/R$ is a Cartan subgroup of $G/R$. Then $C_SR/R=CR/R\subset MR/R$. Therefore, $MR=C_SR$, and hence 
 $MR$ is closed and $M=C_S(M\cap R)$. Also, from Proposition \ref{cetralizer-cartan}\,(3), we have that $R=C_ZN$. Therefore, 
 $M\cap R=C_Z(M\cap N)$ and $M=C_SC_Z(M\cap N)=C(M\cap N)$.
	
Let $N'=N_M(C)$, the normalizer of $C=C_SC_Z$ in $M$.   Then $N'=C_SC_Z(N'\cap N)\subset CN_N(C)$. 
It is enough to show that $N'=C$, or more generally, that $N_N(C)=C_Z\cap N$. Then it would follow from 
Lemma \ref{normalizer}\,(1) that $M=C$. As $C_Z=C\cap R$, it is normal in $N'$. 
 We know from Proposition \ref{cetralizer-cartan}\,(3) that
  $C_Z\cap N$ is connected. It is also a closed normal subgroup of $N_N(C)$. Let $H=C_Z\cap N$ and let $L=N_N(C_Z)$. 
 Then $H$ is a closed connected normal subgroup of $L$,   $H=C\cap N$ and $H\subset N_N(C)\subset L$. Let $T$ be the maximal 
 compact connected central subgroup of $G$. Then $T\subset Z\cap N$, and hence $T\subset C_Z\cap N=H$. As $C_Z$ contains $T$ and 
  $R=C_ZN$, by Proposition \ref{solv-c}\,(1) $L$ is connected, and hence $L/H$ is connected.  As $N/T$ is simply connected and nilpotent, 
  so is $L/H$. Let $x\in N_N(C)$ and $c\in C_S$. As 
  $x^{-1}cxc^{-1}\in  C\cap N=C_Z\cap N$, $cxc^{-1}=xr$ for some $r\in C_Z\cap N=H$, and hence 
$N_N(C)/H\subset Z_{L/H}(C_S)$.   Since $L/H$ is simply connected and nilpotent, we get by 
Lemma \ref{stabilizer} that $Z_{L/H}(C_S)$ is connected. Moreover, we get by Proposition \ref{centralizer-conn} that 
$Z_{L/H}(C_S)=Z^0_L(C_S)/H$, and hence $Z_L(C_S)=Z^0_L(C_S)(C_Z\cap N)=Z^0_L(C_S)$ is connected and 
 $N_N(C)\subset Z^0_L(C_S)\subset Z_N(C_S)$. By Corollary \ref{solv-norm}, we get that $N_Z(C_Z)=C_Z$. As elements of 
$N_N(C)$ normalize $C_Z=C\cap R$ and $N_N(C)$ is contained in $Z=Z_R(C_S)$, we get that 
$N_N(C)=C_Z\cap N$.   This implies that $M=C$ as noted above, and hence we get that $C$ is a maximal nilpotent group. 

Now we show that $N_G(C)/C$ is finite. As $G=S R$, $C_SR/R$ is a Cartan subgroup of $G/R$   and, in particular, it is closed in $G/R$. 
As $N_G(C)R/R \subset N_{G/R}(C_SR/R)$, we get that $C_SR/R$ has finite index in $N_G(C)R/R$.   In particular, $N_G(C)R/R$ is 
closed in $G/R$ and hence $N_G(C)R$ is closed in $G$. Moreover, $N_G(C)/(N_G(C)\cap C_SR)$, being isomorphic to 
$N_G(C)R/C_SR$, is finite. As $N_G(C)\cap R=N_R(C)$,  we get that $N_G(C)/(C_SN_R(C))$ is finite. As $R=C_ZN$ and 
$N_N(C)=C_Z\cap N$, we have that $N_R(C)=C_ZN_N(C)=C_Z$, and hence we have that $N_G(C)/C=N_G(C)/C_SC_Z$ is finite. 
	
Let $L$ be a subgroup of finite index in $C$.   Then $C^0=L^0$ and $N_G(L)\subset N_G(C^0)$.  
We first show that $N_G(C)=N_G(C^0)$.  One way inclusion is obvious.
We know that $Z_G(C^0)$ is normal in $N_G(C^0)$ and  $C_Z=C^0\cap R$ is normal in $N_G(C^0)$. 
By Lemma \ref{cetralizer-cartan1}\,(2), $Z_G(C^0)C_Z=C_SZ(C\cap R)C_Z=C$,   
and hence $C$ is normal in $N_G(C^0)$. Therefore, $N_G(C^0)=N_G(C)$. 

Now from above, $N_G(C^0)/C$ is finite. As $C/L$ is finite and $N_G(L)\subset N_G(C^0)$, we get that $N_G(C^0)/L$ is finite, 
and hence $N_G(L)/L$ is finite. This shows that $C$ is a Cartan subgroup of $G$. 
	
Conversely, let $C$ be a Cartan subgroup of $G$. By Theorem \ref{Wustner}, there exists a Levi subgroup $S$ of $G$ such that 
$G=SR$ and $C=(C\cap S)(C\cap R)$, where $C_S=C\cap S$ is a Cartan subgroup of $S$ and $(C\cap R)\subset Z_R(C_S)$. 
By Proposition \ref{cetralizer-cartan}\,(2), $Z_R(C_S)$ is a closed connected solvable Lie group.
	
We show that $C\cap R$ is a Cartan subgroup of $Z_R(C_S)$. Suppose that $C'$ is a nilpotent subgroup of $Z_R(C_S)$  
such that $C\cap R\subsetneq C'$. Then $C_SC'$ is a nilpotent subgroup of $G$   and $C\subsetneq C_SC'$; this contradicts 
the fact that $C$ is a maximal nilpotent group (as a Cartan subgroup in $G$). This proves that $C\cap R$ is a maximal nilpotent 
subgroup of $Z_R(C_S)$. 
	
We know from Theorem \ref{Wustner} that $C\cap R$ is connected. Let $N''$ be the normalizer of $C\cap R$ in $Z_R(C_S)$. 
Then $C_SN''=CN''$ normalizes $C$. As $C$ is a Cartan subgroup of $G$, $CN''/C$ is finite. Since the latter group is isomorphic 
to $N''/(N''\cap C)$ and $N''\cap C=C\cap R$, we get that $N''/(C\cap R)$ is finite. By \cite[Theorem 1.9\,(iii)]{Wu2}, $C\cap R$ is 
connected, it does not admit any proper (closed normal) subgroup of finite index. Therefore, $C\cap R$ is a Cartan subgroup of 
 $Z_R(C_S)$. \qed

\begin{rem} For a Levi subgroup $S$ in a connected Lie group $G$, and a Cartan subgroup $C_S$ in $G$, we know that 
$C=C_SC_{Z_R(C_S)}$ is a Cartan subgroup of $G$. Since $Z_R(C_S)$ is connected and solvable, its Cartan subgroups are 
conjugates of each other \cite[Proposition 6]{Wi}. Now it follows from Theorem \ref{cartan-constr} that Cartan subgroups of $G$ 
which contain $C_S$ are 
conjugates of each other (by elements in $Z_N(C_S)$). Moreover, it is shown in the proof that $N_N(C)=C\cap N$, $N_R(C)=C\cap R$ 
and that $N_G(C)=N_G(C^0)$. In particular, from the proof of the theorem we can assert the following: $C$ is Cartan subgroup of $G$ 
if and only if $C$ is a maximal nilpotent group and $N_G(C^0)/C$ is finite. \end{rem}

Now we deduce a result for Cartan subalgebras which is analogous to Theorem \ref{cartan-constr}. For a Lie algebra $\G$, 
let the subalgebra $\RR$ be the radical of $\G$ and for a Lie subalgebra $\LL$ of $\G$, 
$\ZZ_\RR(\LL)=\{X\in\RR\mid {\rm ad}(X)(Y)=[X,Y]=0 \mbox{ for all } Y\in\LL\}$ (where $[\ ,\ ]$ denotes the Lie 
bracket operation on $\G$).  

\begin{cor} Let $\G$ be a Lie algebra and let $\Sc$ be any Levi subalgebra such that $\G=\Sc\ltimes\RR$, where $\RR$ is the radical 
of $\G$. Let   $\HH_\Sc$ be a Cartan subalgebra of $\Sc$   and let $\HH_{\ZZ_\RR(\HH_\Sc)}$ be a Cartan subalgebra of 
$\ZZ_\RR(\HH_\Sc)$. Then $\HH=\HH_\Sc\oplus\HH_{\ZZ_\RR(\HH_\Sc)}$ is a Cartan subalgebra of $\G$. Conversely, for 
any Cartan subalgebra $\HH$ of $\G$ there exists a Levi subalgebra $\Sc$ such that 
$\HH=(\HH\cap \Sc)\oplus \HH_{\ZZ_\RR(\HH\cap\Sc)}$, where $\HH_{\ZZ_\RR(\HH\cap\Sc)}=\HH\cap\RR$ is a 
Cartan subalgebra of $\ZZ_\RR(\HH\cap\Sc)$. 
\end{cor}

\begin{proof}
Let $G$ be a connected Lie group whose Lie algebra is $\G$ \cite[Ch.\ XII, Theorem 1.1]{H}. Let $S$ be a Levi subgroup 
of $G$ with the Lie algebra $\Sc$. Then $G=SR$, a Levi decomposition, where $R$ is the radical of $G$. Given a Cartan 
subalgebra $\HH_\Sc$ of $\Sc$ there exists a Cartan subgroup of $S$, namely $C_S:=C(\HH_\Sc)$ (notation as in \S\,2). 
By Proposition \ref{cetralizer-cartan}, $Z_R(C_S)$ is connected. Moreover, its Lie algebra is $\ZZ_\RR(\HH_\Sc)$. 
Let $C_{Z_R(C_S)}=C(\HH_{\ZZ_\RR(\HH_\Sc)})$, which is a Cartan subgroup of $Z_R(C_S)$.  
By Theorem \ref{cartan-constr}, $C=C_SC_{Z_R(C_S)}$   is a Cartan subgroup. 
Therefore, $\HH=\HH_\Sc\oplus\HH_{\ZZ_\RR(\HH_\Sc)}$ is a Lie algebra of $C$. Thus $\HH$ is a Cartan subalgebra of $\G$. 

The converse can be proven easily by using W\"ustner's decomposition theorems \cite[Theorems 1.8 and 1.11]{Wu2} and the 
converse statement in Theorem \ref{cartan-constr}. \qed
\end{proof}

The next lemma, which will be useful in proving Theorems \ref{str-solv} and \ref{quo-cartan}, follows easily 
from the definition of Cartan subgroups given by Chevalley; we omit the proof. 
	
\begin{lem} \label{sub-cartan}
Let $G$ be a connected Lie group and let $C$ be a Cartan subgroup of $G$. Let $H$ be a closed connected subgroup 
such that $C\subset H$. Then $C$ is a Cartan subgroup of $H$. 
\end{lem}

\medskip
\noindent{\it Proof of Theorem \ref{str-solv}.} Let $G$ be a connected Lie group with the radical $R$ and let $T_R$ be a 
maximal compact subgroup of $G$. If $T_R$ is central in $G$,  then $T_R$ is the unique 
maximal compact subgroup of $R$ and $Z_G(T_R)=G$ is connected and the first assertion holds trivially. Now suppose 
$G$ is such that $T_R$ is nontrivial and it is not central in $G$. We first want to show that $Z_G(T_R)$ is connected. 

\medskip
\noindent {\bf Step 1:} Since $T_R$ is compact and connected, there exists $k$ in $T_R$ which 
generates a dense subgroup in $T_R$ and $\Ad(k)$ is semisimple. Hence $Z_H(T_R)=Z_H(k)$ for any $T_R$-invariant closed 
subgroup $H$ of $G$. As $k$ acts trivially on $G/R$ and $R/N$, and $\Ad(k)$ is semisimple, by Proposition \ref{centralizer-conn} 
and the Remark \ref{rem1}, we have that $G=Z_G(T_R)Z_R(T_R)=Z_G^0(T_R)Z_R(T_R)$ and $Z_R(T_R)=Z^0_R(T_R)Z_N(T_R)$. 
In particular, $Z_G(T_R)=Z^0_G(T_R)Z_N(T_R)$. For the maximal compact connected central subgroup 
$T$ of $G$, we know that $T\subset Z_N(T_R)$ and by Lemma \ref{stabilizer}, $Z_{N/T}(k)$ is connected. Again it follows from 
Proposition \ref{centralizer-conn} and Remark \ref{rem1} that $Z_N(k)=Z_N(T_R)$ is connected. 
Therefore, $Z_G(T_R)$ is connected.

Note that $G=Z_G(T_R)N=Z_G(T_R)R$, all the Levi subgroups of $Z_G(T_R)$ are Levi subgroups of $G$ and for a Levi subgroup
$S$ of $G$ contained in $Z_G(T_R)$, we have that $Z_G(T_R)=SZ_R(T_R)$. Therefore, $Z_R(T_R)$ is the radical of $Z_G(T_R)$. 

\medskip
\noindent{\bf Step 2:} Now we show that any Cartan subgroup of $Z_G(T_R)$ is a Cartan subgroup of $G$. Let $G'=Z_G(T_R)$ 
and let $R':=Z_R(T_R)$. Then $R'$ is the radical of $G'$ and $T_R$ is the unique maximal compact subgroup of $R'$ and it is 
central in $G'$. 

 Let $C'$ be a Cartan subgroup of $G'$. By Theorem \ref{cartan-constr}, there exists a Levi subgroup 
$S$ of $G'$ such that  $G'=SR'$, $C_S=C'\cap S$ is a Cartan subgroup of $S$ and $C'\cap R'$ is a 
Cartan subgroup of $Z_{R'}(C_S)$, the centralizer of $C_S$ in $R'$. Let $Z':=Z_{R'}(C_S)$  and $C_{Z'}:=C'\cap R'$, 
which is Cartan subgroup of $Z'$. Let $Z:=Z_R(C_S)$. Then $Z'\subset Z$. As $S$ is also a Levi subgroup of $G$, 
$Z$ is connected.  Since $T_R$ centralizes $C_S$, $T_R$ is a maximal compact subgroup of $Z$. 
 If $T_R$ is central in $Z$, then $Z\subset R'$, hence $Z=Z'$, and 
 by Theorem \ref{cartan-constr}, $C'$ is a Cartan subgroup of $G$.  

Suppose $T_R$ is not central in $Z$. Let $N'$ be the nilradical of $Z$. 
Arguing as in Step 1 for $Z$ instead of $G$, we get that $Z=Z_Z(T_R)N'$, where $Z_Z(T_R)$ is the centralizer of $T_R$ in $Z$. 
 Therefore, $Z_Z(T_R)=Z_R(C_S)\cap Z_R(T_R)=Z_{R'}(C_S)=Z'$. Now $Z=Z'N'$. By Proposition 
\ref{solv-c}\,(3), there exists a Cartan subgroup $C_Z$ of $Z$ such that $C_{Z'}=C_Z\cap Z'$. As $T_R\subset C_{Z'}\subset C_Z$
and $C_Z$ is connected and nilpotent, we get that $C_Z\subset Z_R(T_R)=R'$, and hence $C_{Z'}\subset C_Z\subset Z\cap R'=Z'$. 
This implies that $C_{Z'}=C_Z$ as the former is a Cartan subgroup of $Z'$. In particular $C_{Z'}$ is a Cartan subgroup of 
$Z=Z_R(C_S)$.  As $S$ is also a Levi subgroup of $G$, by Theorem \ref{cartan-constr}, $C'$ is a Cartan subgroup of $G$. 

Conversely, suppose $C$ is any Cartan subgroup of $G$. We know that $C\cap R$ is connected and nilpotent. Let 
$T_C$ be the unique compact (central) subgroup of $C\cap R$. By \cite[Theorem 3\,(ii)]{Go}, it follows that $T_C$ is 
a maximal compact subgroup of $R$. Now we have from Step 1 that $Z_G(T_C)$ is connected. As $C\cap R$ centralizes 
$C_S$, we get that $T_C$ is central in $C$, i.e.\ $C\subset Z_G(T_C)$. By Lemma \ref{sub-cartan}, 
$C$ is a Cartan subgroup of $Z_G(T_C)$. Now the proof is complete with $K=T_C$. \qed 

\medskip 
Using Theorems \ref{cartan-constr} and \ref{str-solv}, we prove Corollary \ref{cartan-whole} which is a more refined version 
of Theorem \ref{cartan-constr} and describes Cartan subgroups in more detail.

\medskip
\noindent{\it Proof of Corollary \ref{cartan-whole}}. Let $G=SR$ be a Levi decomposition with a Levi subgroup $S$ and the radical 
$R$. Let $C_S$ be a Cartan subgroup of $G$. 
By Theorem \ref{cartan-constr}, $C=C_SC_Z$ is a Cartan subgroup of $G$ such that $C_S=C\cap S$, 
$C_Z=C\cap R$ is a Cartan subgroup of $Z=Z_R(C_S)$ and both $Z$ and $C_Z$ are connected.
By Theorem \ref{str-solv}, for a maximal compact subgroup $T_R$ of $R$ 
contained in $C_Z$, we get that $Z_Z(T_R)$ is connected and $C_Z$ is a Cartan subgroup of $Z_Z(T_R)$. 
But $Z_Z(T_R)=Z_R(C_S)\cap Z_R(T_R)=Z_R(C_ST_R)$. Hence, $Z_R(C_ST_R)$ is connected and $C_Z$ is a 
Cartan subgroup of it. Therefore, $C=C_SC_{Z_R(C_ST_R)}$ and the first assertion in the corollary is proved. 

The converse statement follows from Theorem \ref{cartan-constr} and the proof of the first assertion above. \qed

\medskip
\noindent{\it Proof of Corollary \ref{solv-cartan}}. Let $G$ and $R$ be as in the hypothesis and let $C_S$ be a Cartan 
subgroup of a Levi subgroup $S$. By Theorem \ref{cartan-constr}, $Z_R(C_S)$ is connected and for a 
Cartan subgroup $C_{Z_R(C_S)}$ of $Z_R(C_S)$, $C=C_SC_{Z_R(C_S)}$ is a Cartan subgroup of $G$. 
By Proposition \ref{cetralizer-cartan}\,(2--3), $R=Z_R(C_S)N=C_{Z_R(C_S)}N$. Now from Proposition \ref{solv-c}\,(3), 
there exists a Cartan subgroup $C_R$ of the radical $R$ such that $C_{Z_R(C_S)}=Z_R(C_S)\cap C_R=C\cap C_R$. 
This completes the proof of the first assertion in the corollary. 

The converse statement follows from Theorem \ref{cartan-constr} and the proof of the first assertion above.
 \qed

\begin{rem} We know that Cartan subgroups of a connected solvable Lie group are conjugate to each other \cite[Proposition 6]{Wi}. 
Moreover, in a connected Lie group $G$, conjugates of a Cartan subgroup are Cartan subgroups. Therefore, using 
Corollary \ref{solv-cartan}, we can deduce that given any Cartan subgroup $C_R$ of $R$, there exists a Cartan subgroup $C$ of $G$ 
such that $C\cap R=C\cap C_R$.
\end{rem}
	
Now we state the following corollary which 
is a special case of \cite[Theorem 1.9\,(ii)]{Wu3}; as any connected Lie group $G$, whose radical $R$ is a compact extension of its 
nilradical $N$, can be expressed as $G=SKN$, where $S$ is a Levi subgroup which centralizes a maximal compact subgroup $K$ of $R$, 
and $G$ is Mal'cev splittable, (see \cite{Wu3} for the definition). 

\begin{cor} \label{solv-cpt-nil} Let $G$ be a connected Lie group such that its radical $R$ is a compact extension of its nilradical $N$. 
Then any Cartan subgroup of $G$ is of the form $C_SKZ_N(C_SK)$ where $C_S$ is a Cartan subgroup of a Levi subgroup $S$ of 
$G$ and $K$ is a maximal compact subgroup of $R$ contained in $Z_R(C_S)$. Moreover, for any Levi subgroup $S$ of 
$G$ and a Cartan subgroup $C_S$ of $S$, $Z_R(C_S)$ contains a maximal compact subgroup $K$ of $R$ and 
$C=C_SZ_R(C_SK)=C_SKZ_N(C_SK)$ is a Cartan subgroup of $G$. 
\end{cor}

The proof follows easily from Corollary \ref{cartan-whole} as $Z_R(C_SK)=KZ_N(C_SK)$, and since it is (connected and) nilpotent, 
$C_{Z_R(C_SK)}=Z_R(C_SK)=KZ_N(C_SK)$.  	
	
Note that a connected nilpotent Lie group has a unique Cartan subgroup, namely, $G$ itself. If the radical of $G$ is nilpotent, then its maximal compact subgroup is central in $G$. Hence, the following corollary is an immediate 
consequence of Corollary \ref{solv-cpt-nil}. 
	
\begin{cor} \label{solv-nil} Let $G$ be a connected Lie group such that its radical is nilpotent. Then any Cartan subgroup is of
the form $C_SZ_N(C_S)$, where $C_S$ is a Cartan subgroup of a Levi subgroup $S$ and $N$ is the nilradical of $G$. Moreover,
for any Cartan subgroup $C_S$ of a Levi subgroup $S$, there is a unique Cartan subgroup of $G$ containing $C_S$. 
\end{cor}

\section{Cartan subgroups in quotients of a Lie group}
	
In this section, we prove Theorem \ref{quo-cartan} which asserts that for any closed normal subgroup $H$ of $G$, 
Cartan subgroups of $G/H$ are precisely the images of Cartan subgroups of $G$. Note that this is known in the case 
when $H$ is central. Note also that from analogous results about Cartan subalgebras in \cite[Ch.\ VII, \S\,2]{Bou}, we can 
deduce the above for the connected component of the identity in Cartan subgroups of $G/H$. 
However, the theorem for the general case is not known as Cartan subgroups of a connected Lie group need not be connected. 
Here, we give a group theoretic proof using Chevalley's criterion for Cartan subgroups and results in \S\,2 and \S\,3 above.   

\medskip
\noindent{\it Proof of Theorem~\ref{quo-cartan}}. Let $G$ be a connected Lie group, $H$ be a closed normal subgroup of $G$ 
and let $\pi:G\to G/H$ be the natural projection. If $H$ is central subgroup of $G$, then it is contained in a Cartan subgroup of 
$G$, and $C$ is a Cartan subgroup of $G$ if and only if $\pi(C)$ is a Cartan subgroup of $G/H$. Note that $H/H^0$ is a closed 
central subgroup of $G/H^0$. Hence to prove (a) and (b), we may assume that $H$ is connected. 

\medskip
\noindent{\bf Step 1:} Now suppose $H$ is semisimple. Here, $Z(H)$ is a discrete normal subgroup of $G$, and hence it is 
central in $G$. As $H$ is connected $H/Z(H)$ has trivial center, and hence we may replace $G$ and $H$ by $G/Z(H)$ and $H/Z(H)$ 
respectively and assume that the center of $H$ is trivial. By \cite[Lemma 3.9]{Iw} (and its proof), we have $G=HZ_G(H)=HZ_G^0(H)$. 
Here, $H\cap Z^0_G(H)\subset Z(H)=\{e\}$. Therefore, $G=H\times Z_G^0(H)$, as $G$ and $H$ are connected. Since any Cartan 
subgroup of $G$ is a direct product of a Cartan subgroup of $H$ and a Cartan subgroup of $Z^0_G(H)$, it is easy to see that 
(a) and (b) hold in this case. 

\medskip
\noindent{\bf Step 2:} Note that the radical $R_H$ of $H$ is a connected solvable characteristic subgroup in $H$, and hence it is 
contained in the radical $R$ of $G$. Now $H/R_H$ is a connected semisimple subgroup in $G/R_H$. Hence from Step 1, we get that 
(a) and (b) hold for $G/R_H$ and $H/R_H$ instead of $G$ and $H$. Therefore, we may assume that $H$ is solvable. 

\medskip
\noindent{\bf Step 3:}  We show that it may be assumed that $H\cap Z(G)$ is trivial, where $Z(G)$ is the center of $G$. Note that 
$H\cap Z(G)$ is a closed central subgroup of $G$. If the connected 
component of the identity $e$ in $H\cap Z(G)$ is nontrivial, then the dimension of $H/(H\cap Z(G))$ is strictly less than that of $H$. 
As $H\cap Z(G)$ is central, as observed earlier, we may replace $G$ and $H$ by $G/(H\cap Z(G))$ and $H/(H\cap Z(G))$ respectively. 
Repeating this process finitely 
many times,  we get that $H\cap Z(G)$ is discrete, as the dimension of $G$ is finite. Since $G$ is connected, it follows that 
$H/(H\cap Z(G))$ has trivial intersection with the center of $G/(H\cap Z(G))$. Thus, we may replace $H$ by $H/(H\cap Z(G))$ 
and $G$ by $G/(H\cap Z(G))$ and assume that $H\cap Z(G)$ is trivial.  

Suppose $H=R$. Then for any Levi subgroup $S$ of $G$, we have that $S\cap R=S\cap H$ is   discrete and central in $S$, and 
$G/R$ is isomorphic to $S/(S\cap R)$.   By Theorem \ref{Wustner}, any Cartan subgroup $C$ has the form $C=(C\cap S)(C\cap R)$ 
for some Levi subgroup $S$ and $C\cap S$ is a Cartan subgroup of $S$. Now $(C\cap S)/(S\cap R)$ is a Cartan subgroup of 
$S/(S\cap R)$. Therefore, $\pi(C)$ is a Cartan subgroup of $G/H$ and (a) holds. Conversely, if $C'$ is a Cartan subgroup of $G/R$, 
then since $G/R$ is isomorphic to   $S/(S\cap R)$ for any Levi subgroup $S$, we may choose a semisimple subgroup $S$ and a 
Cartan subgroup $C_S$ such that $\pi(C_S)=C'$. By Theorem \ref{cartan-constr}, there exists a Cartan subgroup 
$C=C_SC_{Z_R(C_S)}$ of $G$, where $C_{Z_R(C_S)}$ is any Cartan subgroup of $Z_R(C_S)$.  	Now $\pi(C)=\pi(C_S)=C'$ and 
(b) holds in this case. 

\medskip
\noindent{\bf Step 4:} Now suppose $H\subsetneq R$ is solvable. Since $H$ is connected and solvable, there exists 
a sequence of closed connected normal  subgroups $H=H_0\supset H_1\supset\cdots\supset H_k=\{e\}$ in $G$ 
such that for $i=1,\ldots, k$, $H_{i-1}/H_i$ is a closed connected normal abelian 
subgroup of $G/H_i$. Hence we may assume that $H$ is a closed connected normal abelian subgroup of $G$, 
i.e.\ $H\subset N$, the nilradical 
of $G$. As shown in Step 3, we may also assume that $H\cap Z(G)$ is trivial.   In particular, the largest compact normal subgroup $K$ 
of $H$ is abelian and it is normal in $G$, and hence it is central in $G$. Therefore, $K$ is trivial and $H$ is a simply connected abelian 
subgroup. Now $G/H=\pi(S)\pi(R)$ is a Levi decomposition for any Levi subgroup $S$ of $G$. Here, $S\cap H\subset S\cap R$ is 
discrete and central in $S$. Therefore, $\pi(S)$ is isomorphic to $S/(S\cap H)$. Let $C=C_SC_Z$ be a Cartan subgroup of $G$, 
where $C_Z$ is a Cartan subgroup of $Z=Z_R(C_S)$. Then $\pi(C_S)$ is a Cartan subgroup of $\pi(S)$. By Theorem \ref{cartan-constr}, 
to prove (a), it is enough to show that $\pi(C_Z)$ is a Cartan subgroup of $Z':=Z_{R/H}(C_S)$. Since $\pi(C_S)$ is a Cartan subgroup of 
$\pi(S)$, by Proposition \ref{cetralizer-cartan}, $Z'$ is connected. Now by Proposition \ref{centralizer-conn}, we have that $\pi(Z)=Z'$. 
In particular, $ZH$ is closed, connected and $Z'=ZH/H$ is isomorphic to $Z/(Z\cap H)$. 
Since $H$ is simply connected and abelian, $Z\cap H=Z_H(C_S)$ is connected by Lemma \ref{stabilizer}. 

Now to prove (a), we may replace $G$ by $Z$ and $H$ by $Z\cap H$, and assume that $G$ is solvable. We know that $G=CN$, 
where $N$ is the nilradical of $G$. Then $\pi(G)=\pi(C)\pi(N)$ and $\pi(N)\subset N'$, where $N'$ is the nilradical of $G/H$. 
Since $\pi(C)$ is nilpotent, so is $\ol{\pi(C)}$, and by Proposition \ref{solv-c}\,(2), we get that there exists a Cartan subgroup $C'$ of 
$G/H$ such that $\ol{\pi(C)}\subset C'$. Now we show that $\pi(C)=C'$. Note that $C'$ is connected. Let $M=\pi^{-1}(C')$. Then 
$M$ is connected and $C\subset M$, and by Lemma \ref{sub-cartan}, $C$ is a Cartan subgroup of $M$.  By \cite[Lemma 9]{Wi}, 
$CM_k=M$ for all $k\in\N$, where $M_1=\ol{[M,M]}$ and $M_{k+1}=\ol{[M, M_k]}$, $k\in\N$. Since $\pi(M)=C'$ is nilpotent, we get 
that $M_k\subset H$ for some $k$, hence $CH=M$, and $\pi(C)=C'$. Therefore (a) holds. 

\medskip
\noindent{\bf Step 5:} Now to complete the proof of (b), we need to prove (b) for the case when $H$ is connected, abelian and 
$H\cap Z(G)=\{e\}$. Let $C'$ be a Cartan subgroup of $G/H$. Let $S'R'$ be the Levi decomposition of $G/H$, where $R'$ is the 
radical of $G/H$ and $S'$ is a Levi subgroup of $G/H$, such that $C_{S'}=C'\cap S'$ and $C'\cap R'=C_{Z'}$, where $Z'=Z_{R'}(C_{S'})$. 
Note that as $H\subset N$,  $R'=\pi(R)$. Now it is easy to see that there exists a Levi subgroup $S$ in $G$ such that $\pi(S)=S'$; 
(this follows from the fact that the image of a Levi decomposition of $G$ under $\pi$ is a Levi decomposition of $G/H$ and all the 
Levi subgroups are conjugate to each other). Here, $S\cap H\subset S\cap R$ is discrete and central in $S$. Hence $S/(S\cap H)$ is 
isomorphic to $S'$ under $\pi$ and there exists a Cartan subgroup $C_S$ of $S$ such that $\pi(C_S)=C_{S'}$. We know from 
Proposition \ref{centralizer-conn} that $\pi(Z_R(C_S))=Z_{R/H}(C_{S'})$. 
Let $Z=Z_R(C_S)$. Now it is enough to show that there exists a Cartan subgroup $C_Z$ in $Z$ such that $\pi(C_Z)=C_{Z'}$ as 
this would imply that for $C=C_SC_Z$ and $\pi(C)=C'$. This together with Theorem \ref{cartan-constr} would imply that $C$ is a 
Cartan subgroup of $G$. As noted in Step 4, $ZH$ is closed, $ZH/H=Z'$, $Z'$ is isomorphic to   $Z/(Z\cap H)$ and 
$Z\cap H=Z_H(C_S)$ is connected. Now we may replace $G$ by $ZH$ and assume that $G$ is solvable and $C'$ is a 
Cartan subgroup of $G/H$. 
Recall that $H$ is connected, normal and abelian and $H\subset N$, the nilradical of $G$. Let $M=\pi^{-1}(C')$. Then $M$ is connected 
and $\pi(M)=C'$. Let $C_M$ be a Cartan subgroup of $M$. Now we show that $\pi(C_M)=C'$. For $M_k$ defined as above, we have 
that $C_MM_k=M$, $k\in\N$, and as $C'$ is nilpotent, $C_MH=M$. Therefore, $\pi(C_M)=C'$.   As $C'$ is a Cartan subgroup of
$G/H$, we get that $C'L'=G/H$, where $L'$ is the closure of the commutator subgroup of $G/H$. Let $L:=\pi^{-1}(L')$. Then 
$L=\ol{[G,G]H}\subset N$. Also, since $\pi(G)=C'L'$, we have that $G=ML\subset MN$. By Proposition \ref{solv-c}\,(3), there exists 
a Cartan subgroup $C$ of $G$, such that $C\cap M=C_M$. Now $\pi(C)$ is nilpotent and it contains $C'$. Therefore, $\pi(C)=C'$. 
This implies that (b) holds. \qed

\smallskip
Using Theorems \ref{cartan-constr} and \ref{quo-cartan} and Propositions \ref{solv-c} and \ref{cetralizer-cartan}, we get the following 
corollary in which the first statement generalises 
a part of Corollary \ref{solv-cartan} and the first part of the second statement generalises Theorem 1.9 of \cite {Wu2}. 

\begin{cor} \label{cartan-subh} Let $G$ be a connected Lie group and let $H$ be a closed connected normal subgroup of $G$. 
Let $C$ be a Cartan subgroup of $G$. Then the following hold:
\begin{enumerate}
\item $C\cap H$ is contained in a Cartan subgroup of $H$.
\item If $H$ is solvable, then $C\cap H$ is connected and $H=(C\cap H)N_H$, where $N_H$ is the nilradical of $H$. 
\end{enumerate}
\end{cor}

\begin{proof} Let $R_H$ and $N_H$ be respectively the radical and nilradical of $H$. They are characteristic in $H$ and hence 
normal in $G$. Therefore $R_H\subset R$, the radical of $G$ and $N_H\subset N$, the nilradical of $G$.

We first prove (2). Let $H$ be solvable, then $H=R_H\subset R$. Note that from Corollary \ref{solv-cartan}, 
there exists a Cartan subgroup $C_R$ of $R$ such that $C\cap R=C\cap C_R$. Then $C\cap H=C_R\cap H$. 
Therefore, to prove both the statements in (2), we may replace $G$ by $R$ and $C$ by 
$C_R$ and assume that $G$ is solvable.

By Theorem \ref{quo-cartan}, $CH/H$ is a Cartan subgroup of $G/H$ which is connected. Therefore, $CH$ is closed
and connected, and by Lemma \ref{sub-cartan}, $C$ is a Cartan subgroup of $CH$. Hence, without loss of any generality we may
assume that $G=CH$.

From \cite[Theorem 1.9\,(i)]{Wu2}, we know that $C\cap N_H$ is connected. Note that 
$CN_H$ is connected. Also, from Theorem \ref{quo-cartan}, we have that $CN_H/N_H$ is a Cartan subgroup of 
$G/N_H$, and hence $CN_H$ is closed. As $H/N_H$ is a closed connected abelian subgroup of $G/N_H$, by 
\cite[Theorem 1.9\,(i)]{Wu2}, $((CN_H)/N_H)\cap (H/N_H)$ is connected. Therefore, $((C\cap H)N_H)/N_H$ is connected. 
As the later group is isomorphic to $(C\cap H)/(C\cap N_H)$ and $C\cap N_H$ is connected, we get that $C\cap H$ is connected.

Now we want to prove that $(C\cap H)N_H=H$. Note that $H$ is solvable and, as assumed above, $G$ is also solvable. 
By Theorem \ref{quo-cartan}, $CH/H$ is a Cartan subgroup of $G/H$ which is connected. Therefore, $CH$ is closed and 
connected, and by Lemma \ref{sub-cartan}, $C$ is a 
Cartan subgroup of $CH$. Hence without loss of any generality, we may assume that $G=CH$. Now 
$[G,G]=[C,C](H\cap [G,G])\subset [C,C](H\cap N)$. Since $(H\cap N)^0=N_H$, $(H\cap N)/N_H$, is a discrete 
subgroup of $G/N_H$, and hence it is central in $G/N_H$ and contained in its Cartan subgroup $CN_H/N_H$. Therefore, 
$H\cap N\subset CN_H$ which is closed, and hence $\ol{[G,G]}\subset CN_H$. By \cite[Lemma 9]{Wu1}, 
$G=C\ol{[G,G]}=CN_H$. Therefore, $H=H\cap CN_H=(C\cap H)N_H$. 

Now we prove (1). If $H$ is nilpotent, then $H$ contains $C\cap H$ and $H$ itself is a Cartan subgroup of $H$. Now suppose 
$H$ is solvable but not nilpotent. From (2), $H=(C\cap H)N_H$. 
As $(C\cap H)$ is nilpotent, by Proposition \ref{solv-c}\,(2), $C\cap H$ is contained in a Cartan subgroup of $H$. 

Suppose $H$ is semisimple.  Let $L=Z_G^0(H)$. By \cite[Lemma 3.9]{Iw}, $G=HL$, an almost direct product, 
since $D:=H\cap L$ is a discrete central subgroup of $G$. Then $C=C_HC_L$, where $C_H$ (resp.\ $C_L$) is a Cartan subgroup 
of $H$ (resp.\ $L$). Note that $D\subset C_H\cap C_L$ and $C_H=C\cap H$.

Suppose $H$ is neither semisimple nor solvable. Note that $HR$ is normal in $G$ and it follows from a theorem  
in \cite{Rag} that $HR$ is closed. We first show that $C\cap HR$ is contained in a Cartan subgroup of $HR$. Let
$G=SR$ be a Levi decomposition such that $C=C_S(C\cap R)$, where $C_S=C\cap S$ is a Cartan subgroup of $S$ and 
$C\cap R$ is a Cartan subgroup of $Z_R(C_S)$. Let $M=(S\cap HR)^0$. Then $M$ is normal in $S$ and hence it is
closed in $S$ (cf.\ \cite{Rag}). Now $HR=MR$ where $M$ is a semisimple Levi subgroup of $HR$. Note that $S=ML$, 
an almost direct product,
where $L=Z^0_S(M)$ and $L\cap M$ is a discrete central subgroup of $S$. $C_S=C_MC_L=C_LC_M$, where $C_M$ 
(resp.\ $C_L$) is a Cartan subgroup of $M$ (resp.\ $L$). Note that $C\cap R\subset Z_R(C_S)\subset Z_R(C_M)$. 
As $C_M$ is a Cartan subgroup of $M$, by Proposition  \ref{cetralizer-cartan}, both $Z_R(C_M)$ and $Z_N(C_M)$ are connected. 
We also have by Proposition \ref{cetralizer-cartan}\,(3), that $R=(C\cap R)N$, and hence that 
$Z_R(C_M)=(C\cap R)Z_N(C_M)=(C\cap R)N'$ where $N'$ is the nilradical  of $Z_R(C_M)$ and $Z_N(C_M)\subset N'$. As $C\cap R$ is 
nilpotent, by Proposition \ref{solv-c}\,(2), we get that there exists a Cartan subgroup $C'$ of 
$Z_R(C_M)$ such that $C\cap R\subset C'$. By Theorem \ref{cartan-constr}, $C_M C'$ is a Cartan subgroup of $MR=HR$. 
Now $C\cap HR=(C_S(C\cap R))\cap MR=C_M(C\cap R)\subset C_MC'$. That is, $C\cap HR$ is a contained in $C_MC'$ which is a 
Cartan subgroup of $HR$. 

Note that $C\cap H=(C\cap HR)\cap H$ and $HR$ is a closed connected (normal) subgroup of $G$. Since $C\cap HR$ is
contained in a Cartan subgroup of $HR$, to prove that $C\cap H$ is contained in a Cartan subgroup of $H$, 
we assume that $G=HR$, i.e.\ $S\subset HR$. Note that $S\cap H$ is a closed normal subgroup of $S$. 
Let $S_H$ be a Levi subgroup of $H$. Then $G=S_HR$, and hence $S_H$ is a Levi subgroup of $G$. 
Therefore, $S_H$ is conjugate to $S$. Now since $H$ is normal in $G$, $S\subset H$ and hence 
$H=SR_H$. Therefore, $C_S$ is a Cartan subgroup of the Levi subgroup $S$ of $H$. Since $R_H$ is normal in $R$, 
$Z_{R_H}(C_S)$ is a closed normal subgroup of $Z_R(C_S)$. Also, by Proposition \ref{cetralizer-cartan}\,(2), 
$Z_{R_H}(C_S)$ is connected. Note that by Theorem \ref{cartan-constr}, $C\cap R$ is a Cartan subgroup of $Z_R(C_S)$ and 
the latter is solvable. Therefore, from above, $C\cap R_H=(C\cap R)\cap Z_{R_H}(C_S)$ is contained in a Cartan 
subgroup $C''$ of $Z_{R_H}(C_S)$. Now by Theorem \ref{cartan-constr}, 
$C_SC''$ is a Cartan subgroup of $H$ and $C\cap H=C_S(C\cap R_H)$ is contained in $C_SC''$.
\end{proof}

\section{Dense images of power maps}
	
In this section, we prove Theorem \ref{power-map} about the dense images of power maps on a connected Lie group. Recall that for a 
group $G$ and any $k\in\N$, the $k$-th power map $P_k:G\to G$ is defined as $P_k(x)=x^k$, $x\in G$. Here, we denote by $P_k$ for 
the $k$-th power map on any group.  

\medskip
\noindent{\it Proof of Theorem~\ref{power-map}}. Let $G$ be a connected Lie group and let $H$ be a closed normal subgroup of $G$. 
Let $k\in\N$ be fixed. Suppose that $P_k(H)$ is dense in $H$ and that $P_k(G/H)$ is dense in $G/H$. We need to show that $P_k(G)$ 
is dense in $G$. 

\medskip
\noindent{\bf Step 1:} Since $P_k(H)$ is dense in $H$, we get that $H$ has finitely many connected components. 
Indeed, $H/H^0$ is a discrete central subgroup in $G/H^0$, and hence it is finitely generated. Therefore, it is either finite 
or it is of the form $F\times \Z^n$ for some $n\in\N$, where $F$ is a finite group. Now, $P_k$ is surjective on $H/H^0$ and 
it implies that $H/H^0$ is finite. 

\medskip	
\noindent{\bf Step 2:}
To prove the assertion, we may assume $H$ to be connected.  From Step 1, $H/H^0$ is a finite central subgroup of $G/H^0$. 
Since $P_k(H)$ is dense in $H$, it implies that $P_k(H/H^0)=H/H^0$ and $P_k(H^0)$ is dense in $H^0$. As $H/H^0$ is central 
in $G/H^0$, $P_k(G/H)$ is dense in $G/H$ implies that $P_k(G/H^0)$ is dense in $G/H^0$. 

\medskip	
\noindent{\bf Step 3:}
Suppose $H$ is a connected semisimple normal subgroup of $G$.   By \cite[Lemma 3.9]{Iw}, we have $G=H\cdot L$ 
(almost direct product) for $L=Z^0_G(H)$ of $G$, where $D=H\cap L$ is a discrete central subgroup of $G$. Note that 
elements of $H$ and $L$ commute with each other   and $G/H$ is isomorphic to $L/D$. The images of both maps $P_k:H\to H$ and 
$P_k:L/D\to L/D$ are dense in $H$ and $L/D$ respectively. Let $C$ be a Cartan subgroup of $G$. Then $D\subset C$ and 
$C=C_H\cdot C_L$ (almost direct product), where $C_H$ (resp.\ $C_L$) is a Cartan subgroup of $H$ (resp.\ $L$). 
Moreover, $D\subset C_H\cap C_L$.   Since, $P_k(H)$ is dense in $H$, by \cite[Theorem 1.1]{B-M}, $P_k(C_H)=C_H$. Note that 
$C_L/D$ is a Cartan subgroup of $L/D$. As $P_k(L/D)$ is dense in $L/D$, by \cite[Theorem 1.1]{B-M}, $P_k(C_L/D)=C_L/D$. 
Let $x\in C$. Let $x_1\in C_H$ and $x_2\in C_L$ such that $x=x_1x_2$. By the above discussion, there exists an element 
$d\in D$ such that $x_2=y_l^kd$, where $y_l\in C_L$. Now $x=x_1x_2=x_1y_l^kd=x_1dy_l^k$.   As $x_1d\in C_H$, there exists 
$y_h\in C_H$ such that $y_h^k=x_1d$. Since, elements of $C_H$ and $C_L$ commute with each other, we have 
$x=y_h^ky_l^k=(y_hy_l)^k$. This shows that $P_k(C)=C$, and hence $P_k(G)$ is dense in $G$ \cite[Theorem 1.1]{B-M}.

\medskip	
\noindent{\bf Step 4:} Let $H_R$ be the radical of $H$. Then $H_R$ is closed and normal in $G$, and hence it is contained in the 
radical $R$ of $G$. Observe that $H/H_R$ is a (connected) semisimple normal subgroup of $G/H_R$ and $G/H$ is isomorphic to 
$(G/H_R)/(H/H_R)$. Since $P_k(H)$ is dense in $H$, we get that $P_k(H/H_R)$ is dense in $H/H_R$. From Step 3, we get that 
$P_k(G/H_R)$ is dense in $G/H_R$. As $H_R\subset R$, it follows that $P_k(G/R)$ is dense in $G/R$. By \cite[Proposition 3.3]{B-M}, 
we have that $P_k(G)$ is dense in $G$. \qed

\medskip	
Note that the converse of Theorem \ref{quo-cartan} is not true even for closed connected normal subgroups as illustrated by 
Example 2.2 in \cite{H-M} of a weakly exponential Lie group which admits a closed connected normal subgroup which is not 
weakly exponential. Now we show that the following well-known result \cite[Corollary 2.1A]{H-M}   can be deduced easily from 
Theorem \ref{power-map} using \cite[Corollary 1.3]{B-M}.
	
\begin{cor}\label{w.e}
Let $G$ be a connected Lie group, and let $H$ be a closed normal subgroup of $G$. If $H$ and $G/H$ are weakly exponential, 
then $G$ is weakly exponential.
\end{cor}

\begin{proof}
Since $H$ is weakly exponential, it is connected. By \cite[Corollary 1.3]{B-M},  $P_k(H)$ is dense in $H$ and 
$P_k(G/H)$ is dense in $G/H$ for all $k\in\mathbb N$. Hence, by Theorem \ref{power-map}, it follows that   $P_k(G)$ is 
dense in $G$ for all $k\in\N$. By \cite[Corollary 1.3]{B-M}, we get that $G$ is weakly exponential. 
\end{proof}

\smallskip
\noindent{\bf Acknowledgements} A.\ Mandal would like to thank Indian Statistical Institute, Delhi Centre, India for a 
post-doctoral fellowship while most of this work was done. R.\ Shah would like to thank I.\ Chatterji and Universit\'e Nice 
Sophia Antipolis, France, for hospitality during a visit in June 2019 while a part of this work was carried out. R.\ Shah would
also like to thank P.\ Chatterjee for some candid comments. The authors would like to thank S.\ G.\ Dani for useful 
comments on the manuscript.

\end{document}